\documentclass[11pt,draft,amsmath,amscd,amsbsy,amssymb]{amsart}
\usepackage{amsmath,amsthm,latexsym,amscd,amsbsy,amssymb}
 \relax

\chardef\bslash=`\\ 

\makeatletter \def\verbatim{\interlinepenalty\@M \@verbatim
\leftskip\@totalleftmargin\advance\leftskip2pc
\frenchspacing\@vobeyspaces \@xverbatim} \makeatother \hfuzz1pc

\makeatletter \def\dgt@k{\dg@DX=-3 \dg@DY=2 \dg@SIZE=3}
\makeatother

\makeatletter \def\dgt@kk{\dg@DX=3 \dg@DY=-1 \dg@SIZE=3}

\theoremstyle{plain} \newtheorem{thm}{Theorem}[section]
\newtheorem{cor}[thm]{Corollary}
\newtheorem{lemma}[thm]{Lemma}
\newtheorem{prop}[thm]{Proposition}

\theoremstyle{definition} \newtheorem{rem}[thm]{Remark}

\newcommand{\RR}{\bar{\mathbb R}}

\newcommand{\umet}{\mathrm{UMET}}
\newcommand{\cumet}{\mathrm{CUMET}}

\usepackage[all]{xy}

\begin{document}

\title[]
{Max-min measures on ultrametric spaces}
\author[Matija Cencelj]{Matija Cencelj}
\address{Institute of Mathematics, Physics and Mechanics, and Faculty of Education,
University of Ljubljana, P.O.B. 2964, Ljubljana, 1001, Slovenia}
\email{matija.cencelj@guest.arnes.si}

\author[D. Repov\v s]{Du\v{s}an~Repov\v{s}}
\address{Faculty of Mathematics and Physics, and Faculty of Education,
University of Ljubljana, P.O.B. 2964, Ljubljana, 1001, Slovenia}
\email{dusan.repovs@guest.arnes.si}

\author[M. Zarichnyi]{ Michael Zarichnyi}
\address{Department of Mechanics and Mathematics,
Lviv National University, Universytetska Str. 1, 79000 Lviv, Ukraine}
\address{Institute of Mathematics,
University of Rzesz\'ow, Rejtana 16 A, 35-310 Rze\-sz\'ow, Poland}
\email{mzar@litech.lviv.ua}

\thanks{}

\date{\today}


\begin{abstract} The ultrametrization of the set of all probability measures of compact support on the ultrametric spaces was first defined by Hartog and de Vink. In this paper we consider a similar construction for the so called max-min measures on  the ultrametric spaces. In particular, we prove that the functors max-min measures and idempotent measures are isomorphic. However, we show that this is not the case for the monads generated by these functors.
\end{abstract}

\maketitle
\section{Introduction}

The ultrametric spaces naturally appear not only in
different parts of mathematics, in particular, in real-valued
analysis, number theory and general topology, but also have applications in biology, physics, theoretical computer science etc (see e.g.
\cite{Ha, RTV, VR}).

The probability measures of compact support on the ultrametric spaces were investigated by different authors. In particular, Hartog and de Vink \cite{Ha} defined an ultrametric on the set of all such measures. The  properties of the obtained construction were established in \cite{HZ} and \cite{VR}.

The aim of this paper is to find analogs of these results for the other classes of measures. We define the so called max-min measures, which play a similar role to that of probability measures in the idempotent mathematics, i.e., the part of mathematics which is obtained by replacing the usual arithmetic operations by idempotent operations (see \cite{ LMS, MS}). The methods and results of idempotent mathematics find numerous applications \cite{A,au,be}.

Note that the max-min measures are non-additive. The class of non-additive measures finds numerous applications, in particular, in mathematical economics, multicriteria decision making, image processing (see, e.g., \cite{GMS}).

In the case of max-min measures, we start with such measures of finite supports; the general case (max-min measures of compact supports) is obtained by passing to the completions.

One of our results shows that the functors of max-min measures and the idempotent measures in the category of ultrametric spaces and nonexpanding maps are isomorphic. However, we show that the monads generated by these functors are not isomorphic.

\section{Preliminaries}
\subsection{Max-min-measures}

By $\RR$ we denote the extended real line, $\RR=\mathbb R\cup\{-\infty,\infty\}$. Let $\wedge$ and $\vee$ denote the operations max and min in $\RR$, respectively. Following the traditions of the idempotent mathematics we denote by $\odot$ the addition (convention $-\infty \odot x=x$ for all $x\in\RR$, $x<\infty$).

Let $X$ be a topological space. As usual, by $C(X)$ we denote the linear space of (real-valued) continuous functions on $X$. The set $C(X)$ is a lattice with respect to the pointwise maximum and minimum and we preserve the notation $\wedge$ and $\vee$ for these operations.

Given $x\in X$, by $\delta_x$ we denote the Dirac measure in $X$ concentrated at $x$.
Given $x_i\in X$ and $\alpha_i\in\RR$, $i=1,\dots,n$, such that $\wedge_{i=1}^n\alpha_i=\infty$,  we denote by $\vee_{i=1}^n\alpha_i\wedge\delta_{x_i}$ the functional on $C(X)$ defined as follows:
$$\vee_{i=1}^n\alpha_i\wedge\delta_{x_i}(\varphi)=\vee_{i=1}^n\alpha_i\wedge\varphi(x_i).$$

Let us denote by $J_\omega(X)$ the set of all such functionals. We call the elements of $J_\omega(X)$ the {\em max-min measures} of finite support on $X$. The term `measure' means nothing but the fact that $\mu=\vee_{i=1}^n\alpha_i\wedge\delta_{x_i}\in J_\omega(X)$ can also be  interpreted as a set function with values in the extended real line: $\mu(A)=\vee\{\alpha_i\mid x_i\in A\}$, for any $A\subset X$.

The {\em support} of $\mu= \vee_{i=1}^n\alpha_i\wedge\delta_{x_i}\in J_\omega(X)$ is the set $$\mathrm{supp}(\mu)=\{x_i\mid i=1,\dots,n,\ \alpha_i>-\infty\}\subset X.$$

For any map $f\colon X\to Y$ of topological spaces, define the map $J_\omega(f)\colon J_\omega(X)\to J_\omega(Y)$ by the formula:
$$J_\omega(f)(\vee_{i=1}^n\alpha_i\wedge\delta_{x_i})=\vee_{i=1}^n\alpha_i\wedge\delta_{f(x_i)}.$$

Let us recall that $I_\omega(X)$ denotes the set of functionals of the form $\vee_i\alpha_i\odot\delta_{x_i}$, where $\alpha_i\in\RR$ and $\vee_i\alpha_i=0$.
If $\varphi\in C(X)$, then $(\vee_i\alpha_i\odot\delta_{x_i})(\varphi)=\vee_i\alpha_i\odot\varphi(x_i)$. See e.g. \cite{Z}, for the theory of spaces  $I_\omega(X)$ (called the spaces of idempotent measures of finite support) as well as related spaces $I(X)$ (called the spaces of idempotent measures of compact support). Recall that the {\em support} of $\mu= \vee_{i=1}^n\alpha_i\odot\delta_{x_i}\in I_\omega(X)$ is the set $$\mathrm{supp}(\mu)=\{x_i\mid i=1,\dots,n,\ \alpha_i>-\infty\}\subset X.$$

\begin{rem} We adopt the following conventions: $+\infty\wedge\delta_x=\delta_x$ in $J_\omega(X)$ and $0\odot\delta_x=\delta_x$ in $I_\omega(X)$.
\end{rem}
\subsection{Ultrametric spaces}

Recall that a metric $d$ on a set $X$ is said to be an {\it
ultrametric} if the following strong triangle inequality holds:
$$d(x,y)\le\max\{d(x,z),d(z,y)\}$$
for all $x,y,z\in X$.

 By $O_r(A)$
we denote the $r$-neighborhood of a set $A$ in a metric space. We
write $O_r(x)$ if $A=\{x\}$. It is well-known that in the ultrametric spaces, for any $r>0$, every two distinct elements of the family $\mathcal O_r=\{O_r(x)\mid x\in X\}$ are disjoint. We denote by $\mathcal F_r$ the set of all functions on $X$ that are constant on the elements of the family $\mathcal O_r$. By $q_r\colon X\to X/\mathcal O_r$ we denote the quotient map. We endow the set $X/\mathcal O_r$ with the quotient metric, $d_r$. It is easy to see that $d_r(O_r(x),O_r(y))=d(x,y)$, for any disjoint $O_r(x)$, $O_r(y)$, and the obtained metric is an ultrametric.

 Recall that a map $f\colon X\to Y$,
where $(X,d)$ and $(Y,\varrho)$ are metric spaces, is called {\it
nonexpanding} if $\varrho(f(x),f(y))\le d(x,y)$, for every $x,y\in
X$.
Note that the quotient map $q_r\colon X\to X/\mathcal O_r$ is nonexpanding.

\subsection{Hyperspaces and symmetric powers}

By $\exp X$ we denote the set of all nonempty compact  subsets in
$X$ endowed with the Hausdorff metric:
$$d_H(A,B)=\inf\{\varepsilon>0\mid A\subset O_\varepsilon(B),\ B\subset O_\varepsilon(A)\}.$$

We say that $\exp X$ is the {\em hyperspace} of $X$.
For a continuous
map $f\colon X\to Y$ the map $\exp f\colon\exp X\to\exp Y$ is
defined as $(\exp f)(A)=f(A)$.

It is well-known that $\exp f$ is a nonexpanding map if so is $f$.
We denote by $s_X\colon X\to \exp X$ the singleton map,
$s_X(x)=\{x\}$.

By $S_n$  we denote the group of permutations of the set
$\{1,2,\dots,n\}$. Every subgroup $G$ of the group $S_n$ acts on the $n$-th
power $X^n$ of the space  $X$ by the permutation of factors. Let
$SP^n_G(X)$ denote the orbit space of this action. By
$[x_1,\dots,x_n]$ (or briefly $[x_i]$) we denote the orbit containing $(x_1,\dots,x_n)\in X^n$.

If $(X,d)$ is a metric space, then  $SP^n_G(X)$ is endowed with the following metric $\tilde d$,
$$\tilde d([x_i],[y_i])=\min\{\max\{d(x_i,y_{\sigma(i)}) \mid i=1,\dots,n\}\mid\sigma\in G\}.$$

It is known that the space $(SP^n_G(X),\tilde d)$ is ultrametric if such is $(X,d)$.

Define the map $\pi_G=\pi_{GX}\colon X^n\to SP^n_G(X)$ by the formula $\pi_G(x_1,\dots,x_n)=[x_1,\dots,x_n]$. It is shown in \cite{HZ} (and easy to see) that the map $\pi_G$ is nonexpanding.

\subsection{Monads}

We recall some necessary definitions from the category theory;
see, e.g., \cite{BW,ML} for details.
A {\it monad} $\mathbb T=(T,\eta,\mu)$ in the category ${\mathcal
E}$ consists of an endofunctor $T\colon {\mathcal E}\to{\mathcal
E}$ and natural transformations $\eta\colon 1_{\mathcal E}\to T$
(unity), $\mu\colon T^2=T\circ T\to T$ (multiplication) satisfying
the relations $\mu\circ T\eta=\mu\circ\eta_T=${\bf 1}$_T$ and
$\mu\circ\mu_T=\mu\circ T\mu$.

Given two monads, $\mathbb T=(T,\eta,\mu)$ and $\mathbb
T'=(T',\eta',\mu')$, we say that a natural transformation
$\alpha\colon T\to T'$ is a {\em morphism} of $\mathbb T$ into
$\mathbb T'$ if $\alpha\eta=\eta'$ and
$\mu'\alpha_TT(\alpha)=\alpha\mu$.

We denote by $\umet$ the category of ultrametric spaces and
nonexpanding maps. One of examples of monads on the category
$\umet$ is the hyperspace monad $\mathbb H=(\exp,s,u)$. The
singleton map $s_X\colon X\to \exp X$ is already defined and the
map $u_X\colon \exp^2X\to\exp X$ is the union map, $u_X(\mathcal
A)=\cup\mathcal A$.

It is well-known (and easy to prove) that the max-metric on the finite product of ultrametric spaces is an ultrametric. We will always endow the product with this ultrametric.

{\it The Kleisli category of a monad $\mathbb T$}
is a category  $\mathcal C_{\mathbb T} $ defined by the conditions: $| {\mathcal C}_{\mathbb T}  |=| \mathcal C  |$,
${\mathcal C}_{\mathbb T}  (X,Y) =\mathcal C (X,T(Y))$, and the composition $g \ast f$ of morphisms
  $f\in {\mathcal C}_{\mathbb T}  (X,Y)$,
$g \in {\mathcal C}_{\mathbb T}  (Y,Z)$ is given by the formula $g \ast f= \mu_Z T(g)  f$.

Define the functor $\Phi_{\mathbb T}\colon \mathcal C \to \mathcal C _{\mathbb T} $ by $$\Phi_{\mathbb T} (X)=X, \quad \Phi_{\mathbb T} (f)=\eta_Y  f, \quad X
\in |
\mathcal C |
,\ f \in \mathcal C (X,Y).$$

A functor $ \overline F \colon\mathcal C _{\mathbb T} \to \mathcal C _{\mathbb T} $ is called
an~{\it extension of the functor $F\colon \mathcal C \to \mathcal C $ on the
Kleisli category
$\mathcal C _{\mathbb T} $} if $\Phi_{\mathbb T} F=\overline F \Phi_{\mathbb T} $.

The proof of the following theorem can be found in \cite{Vi}.

\begin{thm}\label{ext_kleisli}
There exists a bijective correspondence between the ex\-ten\-sions of
functor~$F$ onto the~Kleisli  category
$\mathcal C_{\mathbb T}$ of a monad~${\mathbb T}$
and the natural transformations $\xi\colon FT\to TF$
satisfying
\begin{enumerate}
\item[1)]$\xi F(\eta) =\eta_F$;
\item[2)]$\mu_F T(\xi)\xi_T=\xi F(\mu)$.
\end{enumerate}
\end{thm}
\section{Ultrametric on the set of max-min measures}

Let $(X,d)$ be an ultrametric space. For any $\mu,\nu\in J_\omega(X)$ , let $$\hat d(\mu,\nu)=\inf\{r>0\mid \mu(\varphi)=\nu(\varphi)\text{, for any }\varphi\in C(X)\}.$$
Since $\mu,\nu$ are of finite support, it is easy to see that $\hat d$ is well defined.

\begin{thm} The function $\hat d$ is an ultrametric on the set $J_\omega(X)$.
\end{thm}
\begin{proof} We only have to check the strong triangle inequality. Suppose that $\mu,\nu,\tau\in J_\omega(X)$ and $\hat d(\mu,\tau)<r$, $\hat d(\nu,\tau)<r$. Then, for every $\varphi\in \mathcal F_r$, we have $\mu(\varphi)=\tau(\varphi)=\nu(\varphi)$, whence $\hat d(\mu,\nu)<r$.
\end{proof}

\begin{prop} The map $x\mapsto\delta_x\colon X\to J_\omega(X)$ is an isometric embedding.
\end{prop}
\begin{proof} Let $x,y\in X$ and $d(x,y)<r$. Then for every $\varphi\in \mathcal F_r(X)$, we have $\delta_x(\varphi)=\varphi(x)=\varphi(y)=\delta_y(\varphi)$, whence $\hat d(\delta_x,\delta_y)<r$.
Therefore, $\hat d(\delta_x,\delta_y)\le d(x,y)$. The reverse inequality is simple as well.
\end{proof}

\begin{prop} Let $f\colon X\to Y$ be a nonexpanding map of an ultrametric space $(X,d)$ into an ultrametric space $(Y,\varrho)$. Then the induced map $J_\omega(f)$ is also nonexpanding.
\end{prop}

\begin{proof} Since the map $f$ is nonexpanding, $\varphi f\in \mathcal F_r(X)$, for any $ \varphi\in \mathcal F_r(Y)$.

If $\mu,\nu\in J_\omega(X)$ and $\hat d(\mu,\nu)<r$, then, for every $ \varphi\in \mathcal F_r(Y)$, we have $$J_\omega(f)(\mu)(\varphi)= \mu(\varphi f)= \nu(\varphi f)=J_\omega(f)=J_\omega(f)(\nu)(\varphi) $$
and therefore $\hat \varrho(J_\omega(f)(\mu),J_\omega(f)(\nu))<r$.
\end{proof}

We therefore obtain a functor $J_\omega$ on the category $\umet$.

\begin{prop}\label{p:2} If $\mu,\nu\in J_\omega(X)$, then the following are equivalent:
\begin{enumerate}
\item $\hat d(\mu,\nu)<r$;
\item $J_\omega(q_r)(\mu)=J_\omega(q_r)(\nu)$.
\end{enumerate}
\end{prop}
\begin{proof} 1)$\Rightarrow$2). For every $\varphi\colon X/\mathcal O_r\to\mathbb R$ we have $\varphi q_r\in\mathcal F_r$ and therefore $$J_\omega(q_r)(\mu)=\mu(\varphi q_r)=\nu(\varphi q_r)=J_\omega(q_r)(\nu).$$ Thus, $J_\omega(q_r)(\mu)=J_\omega(q_r)(\nu)$.

 2)$\Rightarrow$1). Let  $\varphi \in\mathcal F_r$, then $\varphi$ factors through $q_r$, i.e. there exists $\psi\colon X\to\mathbb R$ such that $\varphi=\psi q_r$. Then $$\mu(\varphi)=\mu(\psi q_r)=J_\omega(q_r)(\mu)(\varphi)=J_\omega(q_r)(\nu)(\varphi)=\nu(\psi q_r)=\nu(\varphi).$$
 Thus, $\hat d(\mu,\nu)<r$.
\end{proof}

In the sequel, given  a metric space $(X,d)$, we denote also by $d$ the (extended, i.e. taking values in $[0,\infty]$) metric on the set of maps from  a nonempty set $Y$ into $X$ defined by the formula: $d(f,g)=\sup\{d(f(x),g(x)\mid x\in X\}$.

\begin{prop} The functor $J_\omega$ is locally non-expansive, i.e., for every nonexpanding maps $f,g$ of an ultrametric space $(X,d)$ into an ultrametric space $(Y,\varrho)$ we have $\hat\varrho(J_\omega(f),J_\omega(g))\le \varrho(f,g)$.
\end{prop}
\begin{proof} If $\varrho(f,g)=\infty$, then there is nothing to prove. Suppose that $\varrho(f,g)<r<\infty$. Then $q_rf=q_rg$, where $q_r\colon Y\to Y/\mathcal O_r(Y)$ is the quotient map. For every $\mu\in J_\omega(X)$, we obtain $$J_\omega(q_r)J_\omega(f)(\mu)=J_\omega(q_rf)(\mu)=J_\omega(q_rg)(\mu)J_\omega(q_r)J_\omega(g)(\mu)$$ and by Proposition \ref{p:2}, $\hat\varrho(J_\omega(f)(\mu),J_\omega(g)(\mu))<r$.
\end{proof}

\section{Categorical properties}

Let $(X,d)$ be an ultrametric space. Given a function $\varphi\in C(X)$, define $\bar\varphi\colon J_\omega(X)\to \mathbb R$ as follows: $\bar\varphi(\mu)= \mu(\varphi)$.

\begin{prop} If $\varphi\in \mathcal F_r(X)$, then $\bar\varphi\in \mathcal F_r(J_\omega(X))$.
\end{prop}

\begin{proof} Given $\mu,\nu\in J_\omega(X)$ with $\hat d(\mu,\nu)<r$, we see that $\bar\varphi(\mu)= \mu(\varphi)=\nu(\varphi)=\bar\varphi(\nu)$, whence $\bar\varphi\in \mathcal F_r(J_\omega(X))$.
\end{proof}

Let $M\in J_\omega^2(X)$. Define $\xi_X(M)$ by the condition $\xi_X(M)(\varphi)=M(\bar\varphi)$, for any $\varphi\in C(X)$. If $M=\vee_i\alpha_i\wedge\delta_{\mu_i}$ and $\mu_i=\vee_j\beta_{ij}\wedge\delta_{x_{ij}}$, then $$\xi_X(M)=\vee_i\vee_j\alpha_i\wedge\beta_{ij}\wedge\delta_{x_{ij}}.$$

\begin{prop} The map $\xi_X$ is nonexpanding.
\end{prop}
\begin{proof} Let $d$ denote the ultrametric on $X$, then $\hat d$ and $\hat{\hat d}$ denote the ultrametrics on $J_\omega(X)$ and $J_\omega^2(X)$ respectively. Let $M,N\in J_\omega^2(X)$ and $\hat{\hat d}(M,N)<r$, for some $r>0$. Then, for every $\varphi\in \mathcal F(X)$ we obtain $$\xi_X(M)(\varphi)=M(\bar\varphi)=N(\bar\varphi)=\xi_X(N)(\varphi)$$ and therefore $\hat d(\xi_X(M),\xi_X(N))<r$.
\end{proof}

It is easy to verify that the maps $\xi_X$
give rise to
a natural transformation of the functor $J_\omega^2$ to the functor $J_\omega$ in the category $\umet$.

\begin{thm}\label{t:monad} The triple $\mathbb J_\omega=(J_\omega,\delta,\xi)$ is a monad in the category $\umet$.
\end{thm}
\begin{proof} Let $\mu=\vee_i\alpha_i\wedge\delta_{x_i}\in J_\omega(X)$. Then $$\xi_X J_\omega(\delta_X)(\mu)= \xi_X (\vee_i\alpha_i\wedge\delta_{\delta_{x_i}})= \vee_i\alpha_i\wedge\delta_{x_i}=\mu$$
and
$\xi_X \delta_{J_\omega(X)}(\mu)= \xi_X (\delta_\mu)=\mu$. Therefore $\xi J_\omega(\delta)=1_{J_\omega}=\xi \delta_{J_\omega}$.

Let $\mathfrak M=\vee_i\alpha_i\wedge\delta_{M_i}\in J_\omega^3(X)$, where $M_i=\vee_j\beta_{ij}\wedge\delta_{\mu_{ij}}$. Then \begin{align*}\xi_X J_\omega(\xi_X)(\mathfrak M)= & \xi_X(\vee_i\alpha_i\wedge\delta_{\xi_X(M_i)}) = \xi_X(\vee_i\alpha_i\wedge\delta_{\vee_j\beta_{ij}\wedge\mu_{ij}})\\ =&\vee_i\vee_j\alpha_i\wedge\beta_{ij}\wedge\mu_{ij} \\
=&  \vee_i\alpha_i\wedge (\vee_j\beta_{ij}\wedge\delta_{\mu_{ij}}) = \xi_X (\vee_i\alpha_i\wedge M_i)=
\xi_X \xi_{J_\omega(X)}(\mathfrak M)
\end{align*}
and therefore $\xi J_\omega(\xi)=\xi \xi_{J_\omega}$.
\end{proof}

\begin{prop} The spaces $I_\omega(X)$ and $J_\omega(X)$ are isometric.
\end{prop}
\begin{proof}  Define a map $h=h_X\colon I_\omega(X)\to J_\omega(X)$ as follows. Let $\mu=\vee_i\alpha_i\odot\delta_{x_i}\in I_\omega(X)$. Define $h(\mu)=\vee_i-\ln(-\alpha_i)\wedge\delta_{x_i}\in J_\omega(X)$.

Suppose that $\hat d(\mu,\nu)<r$, where $\nu=\vee_j\beta_j\odot\delta_{y_j}\in I_\omega(X)$. For every $x\in X$ and $t\le0$, define $\varphi^x_t\colon X\to\mathbb R$ by the conditions:  $\varphi^x_t(y)=0$ if $y\in B_r(x)$ and  $\varphi^x_t(y)=t$ otherwise.

Then $$\max_{x_i\in B_r(x)}\alpha_i=\lim_{i\to-\infty}\mu(\varphi^x_t)= \lim_{i\to-\infty}\nu(\varphi^x_t)= \max_{y_j\in B_r(x)}\beta_j.$$

If $\varphi\in\mathcal F_r$, then $$\mu(\varphi)=\vee_i\alpha_i\odot\varphi(x_i)= \vee_{x\in X}\vee_{x_i\in B_r(x)}\alpha_i\odot\varphi(x_i)=\vee_{x\in X}\vee_{y_j\in B_r(x)}\beta_j\odot\varphi(y_j)$$
and therefore \begin{align*} h(\mu)(\varphi)=&\vee_i-\ln(-\alpha_i)\wedge\varphi(x_i)=\vee_{x\in X}\vee_{x_i\in B_r(x)}-\ln(-\alpha_i)\wedge\varphi(x_i)\\=& \vee_{x\in X}\vee_{y_j\in B_r(x)}-\ln(-\beta_j)\wedge\varphi(y_j)=h(\nu)(\varphi).\end{align*}

Thus, $\hat d(h(\mu),h(\nu))<r$ and we see that the map $h$ is nonexpanding. One can similarly prove that the inverse map $h^{-1}$ is also nonexpanding.
\end{proof}

\begin{prop} The class $\{h_X\}$ is a natural transformation  of the functor $I_\omega$ to the functor $J_\omega$.
\end{prop}

\begin{proof} Let $f\colon X\to Y$ be a map and $\mu=\vee_i\alpha_i\odot\delta_{x_i}\in I_\omega(X)$. Then \begin{align*}J_\omega(f)h_X(\mu)=&J_\omega(f)(\vee_i-\ln(-\alpha_i)\wedge\delta_{x_i}) = \vee_i-\ln(-\alpha_i)\wedge\delta_{f(x_i)}\\=&h_Y(\vee_i\alpha_i\odot\delta_{f(x_i)})=h_YI_\omega(f)(\mu).\end{align*}
\end{proof}
\begin{cor} The functors $I_\omega$ and $J_\omega$ are isomorphic.
\end{cor}

\begin{rem}\label{r:1} Let $\alpha\colon [-\infty,0]\to[-\infty,\infty]$ be an order-preserving  bijection. Then the maps $g^\alpha_X\colon I_\omega(X)\to J_\omega(X)$ defined by the formula $g^{\alpha}_X(\vee_it_i\odot\delta_{x_i}) = \vee_i \alpha(t_i)\wedge\delta_{x_i}$, determine an isomorphism of the functors $I_\omega$ and $J_\omega$.
\end{rem}
\begin{prop}\label{p:iso} Every isomorphism of the functors $I_\omega$ and $J_\omega$ is of the form $g^\alpha$ (see Remark \ref{r:1}), for some order-preserving   bijection $\alpha\colon [-\infty,0]\to[-\infty,\infty]$.
\end{prop}

\begin{proof} Let $k\colon I_\omega\to J_\omega$ be an isomorphism. Let $X=\{x,y,z\}$, where $x,y,z$ are distinct points. Since the functor isomorphisms preserve the supports,    we obtain $$k_X(t\odot\delta_x\vee t\odot\delta_y\vee \delta_z)=\alpha(t)\wedge\delta_x\vee \alpha(t)\wedge\delta_y\vee \beta(t)\wedge\delta_z,$$ where $\alpha(t)\vee\beta(t)=+\infty$.

We are going to show that $\beta(t)=+\infty$, for every $t\in [-\infty,0]$. First note that $k_X(\delta_x\vee \delta_y\vee \delta_z)=\delta_x\vee \delta_y\vee \delta_z$. Suppose that, for some $t\in(-\infty,0)$, we have $\beta(t)<+\infty$. Denote by $r\colon X\to\{y,z\}$ the retraction that sends $x$ to $z$. Then, since in this case $\alpha(t)=+\infty$, we obtain  $$k_{\{y,z\}}(I_\omega(r)(t\odot\delta_x\vee t\odot\delta_y\vee \delta_z))=k_{\{y,z\}}(t\odot\delta_y\vee \delta_z)=\delta_y\vee \delta_z,$$
which is impossible, because the natural transformations preserve the symmetry with respect to the nontrivial permutation of $\{y,z\}$.

Thus, $$k_X(t\odot\delta_x\vee t\odot\delta_y\vee \delta_z)=\alpha(t)\wedge\delta_x\vee \alpha(t)\wedge\delta_y\vee \delta_z$$ and identifying the points $x$ and $y$ we conclude that  $k_{\{y,z\}}( t\odot\delta_y\vee \delta_z)=(\alpha(t)\wedge\delta_y\vee \delta_z)$. We see therefore that $k=g^\alpha$.

It is clear that $\alpha$ is a bijection of $[-\infty,0]$ onto $[-\infty,\infty]$.
Suppose now that $X=\{x_1,x_2,\dots,x_n\}$, where $x_1,x_2,\dots,x_n$ are distinct points. Let $\mu=\vee_{i=1}^n t_i\odot\delta_{x_i}$ be such that $t_1=0$. Given $i>1$, consider a retraction $r_i\colon X\to\{x_1,x_i\}$ that sends every $x_j$, $j\neq i$, to $x_1$. Then, by what was proved above, $$k_{\{x_1,x_i\}}I_\omega(r_i)(\mu)=k_{\{x_1,x_i\}}(\delta_{x_1}\vee t_i\odot\delta_{x_1})=\delta_{x_1}\vee \alpha(t_i)\wedge\delta_{x_1}=J_\omega(r_i)(k_X(\mu))$$
and collecting the data for all $i>1$ we conclude that $k_X(\mu)=\vee_{i=1}^n\alpha(t_i)\wedge\delta_{x_i}$.

We are going to show that the map $\alpha$ is isotone. Again, let $X=\{x,y,z\}$, where  where $x,y,z$ are distinct points. Suppose that $t_1,t_2\in[-\infty,0]$ and $t_1<t_2$. Then $k_X(t_1\odot\delta_{x}\vee t_2\odot\delta_y\vee\delta_z)=\alpha(t_1)\wedge\delta_{x}\vee \alpha(t_2)\wedge\delta_{y}\vee\delta_z$.

For a retraction $r\colon X\to\{y,z\}$ the retraction that sends $x$ to $y$, we obtain $$I_\omega(r)(t_1\odot\delta_{x}\vee t_2\odot\delta_y\vee\delta_z)=t_2\odot\delta_y\vee\delta_z$$
and therefore
$$I_\omega(r)(\alpha(t_1)\wedge\delta_{x}\vee \alpha(t_2)\wedge\delta_{y}\vee\delta_z)=\alpha(t_2)\wedge\delta_{y}\vee\delta_z,$$
whence we conclude that $\alpha(t_1)<\alpha(t_2)$. This finishes the proof of the proposition.
\end{proof}

\begin{thm} The  monads $\mathbb{I}_\omega$ and $\mathbb{J}_\omega$ are not isomorphic. \end{thm}

\begin{proof} Suppose the contrary and let a natural transformation $h\colon I_\omega\to J_\omega$ be an  isomorphism of $\mathbb{I}_\omega$ and $\mathbb{J}_\omega$. Then, by Proposition \ref{p:iso}, $h=g^\alpha$, for some order-preserving   bijection $\alpha\colon [-\infty,0]\to[-\infty,\infty]$.

Let $X=\{a,b,c\}$. Suppose that $M=((-1)\odot\delta_{\mu})\vee\delta_{\nu}\in I_\omega^2(X)$, where $\mu=(-2)\odot\delta_a\vee\delta_b$, $\nu=(-3)\odot\delta_b\vee\delta_c$.

Then \begin{align*}h_X\zeta_X(M)=&h_X((-3)\odot\delta_a \vee(-3)\odot\delta_b\vee\delta_c)\\=&\alpha(-3)\wedge\delta_a \vee \alpha(-3)\wedge\delta_b \vee\delta_c.
\end{align*}
On the other hand, \begin{align*}&\xi_XJ_\omega(h_X)h_{I_\omega(X)}(M)=\xi_XJ_\omega(h_X)
(\alpha(-1)\wedge\delta_{\mu}\vee\delta_{\nu})\\=& \xi_X(\alpha(-1)\wedge\delta_{h_X(\mu)}\vee\delta_{h_X(\nu)})=\xi_X(\alpha(-1)\wedge \delta_{(\alpha(-2)\wedge\delta_a \vee \delta_c)}\vee \delta_{(\alpha(-3)\wedge\delta_b \vee\delta_c)}\\= & (\alpha(-2)\wedge\delta_a \vee\alpha(-3)\wedge\delta_b \vee\delta_c)\neq h_X\zeta_X(M).
\end{align*}
\end{proof}

Let $\mu=\vee_i\alpha_i\wedge\delta_{x_i}\in J_\omega(X)$, $\nu=\vee_j\beta_j\wedge\delta_{y_j}\in J_\omega(Y)$. Define $\mu\otimes\nu\in J_\omega(X\times Y)$ by the formula:
$$\mu\otimes\nu=\vee_{ij}(\alpha_i\vee\beta_j)\wedge\delta_{(x_i,y_j)}.$$
\begin{lemma}\label{l:tensor} The map $$(\mu,\nu)\mapsto \mu\otimes\nu\colon J_\omega(X)\times J_\omega(Y)\to J_\omega(X\times Y)$$ is nonexpanding.
\end{lemma}
\begin{proof} Suppose that $\hat d((\mu,\nu),(\mu',\nu'))<r$. Then $$J_\omega(q_r)(\mu\otimes\nu)=J_\omega(q_r)(\mu)\otimes J_\omega(q_r)(\nu)=J_\omega(q_r)(\mu')\otimes J_\omega(q_r)(\nu')=J_\omega(q_r)(\mu'\otimes\nu')$$ and we conclude that $$\hat d(J_\omega(q_r)(\mu\otimes\nu),J_\omega(q_r)(\mu'\otimes\nu'))<r.$$
Therefore, the mentioned map is nonexpanding.
\end{proof}
\begin{rem} The results concerning the operation $\otimes$ can be easily extended over the products of arbitrary number of factors.
\end{rem}

\begin{thm}\label{t:klei} There exists an extension of the symmetric power functor $SP^n$ onto the category of ultrametric spaces and nonexpanding maps with values that are max-min measures of finite supports.
\end{thm}

\begin{proof} Let $X$ be an ultrametric space. Define a map $\theta_X\colon SP^n_G(J_\omega(X))\to J_\omega(SP^n_G(X))$ by the formula:
$$\theta_X[\mu_1,\dots,\mu_n]= J_\omega(p_G)(\mu_1\otimes\dots\otimes\mu_n).$$

First, we remark that $\theta_X$ is well-defined. Indeed, if $[\mu_1,\dots,\mu_n]=[\nu_1,\dots,\nu_n]$, then there is a permutation $\sigma\in G$ such that $\nu_i=\mu_{\sigma(i)}$, for every $i\in\{1,\dots,n\}$. Denote by $h_\sigma\colon X^n\to X^n$ the map that sends $(x_1,\dots,x_n)$ to $(x_{\sigma(1)},\dots,x_{\sigma(n)})$, then \begin{align*}J_\omega(p_G)(\mu_1\otimes\dots\otimes\mu_n)=&J_\omega(p_Gh_\sigma)(\mu_1\otimes\dots\otimes\mu_n)\\=&
J_\omega(p_G)J_\omega(h_\sigma)(\mu_1\otimes\dots\otimes\mu_n)=J_\omega(p_G)(\nu_1\otimes\dots\otimes\nu_n).
\end{align*}

Next, note that $\theta_X$ is nonexpanding, i.e., a morphism of the category $\umet$. This easily follows from Lemma \ref{l:tensor} and the fact that the map $\pi_G$ is nonexpanding.

Let $(x_1,\dots,x_n)\in X^n$. Then \begin{align*}\theta_XSP^n_G(\delta_X)(x_1,\dots,x_n)=&J_\omega(p_G)(\delta_{x_1}\otimes\dots\otimes\delta_{x_n}) \\ =& J_\omega(p_G)(\delta_{(x_1,\dots,x_n)})=\delta_{p_G(x_1,\dots,x_n)}=\delta_{[x_1,\dots,x_n]}.\end{align*}

Now let  $M_1,\dots,M_n\in J_\omega^2(X)$ and $M_i=\vee\alpha_{ik}\wedge\delta_{\mu_{ik}}$, where $\mu_{ik}\in J_\omega(X)$. Then \begin{align*}\xi_XJ_\omega(\theta_X)\theta_{J_\omega(X)}& ([M_1,\dots,M_n])=\xi_X J_\omega(\theta_X)J_\omega(\pi_{GJ_\omega(X)})(M_1\otimes\dots\otimes M_n)\\ =&J_\omega(\theta_X)J_\omega(\pi_{GJ_\omega(X)})\left(\bigvee(\alpha_{1i_1}\wedge\dots\wedge \alpha_{ni_n})\wedge\delta_{(\mu_{1i_1},\dots,\mu_{ni_n})}\right)\\ =& \mu_X J_\omega(\theta_X)\left(\bigvee(\alpha_{1i_1}\wedge\dots\wedge \alpha_{ni_n})\wedge\delta_{[\mu_{1i_1},\dots,\mu_{ni_n}]}\right) \\ = & \xi_X\left(\bigvee(\alpha_{1i_1}\wedge\dots\wedge \alpha_{ni_n})\wedge\delta_{\theta_X([\mu_{1i_1},\dots,\mu_{ni_n}])}\right) \\ =& \bigvee(\alpha_{1i_1}\wedge\dots\wedge \alpha_{ni_n})\wedge\theta_X([\mu_{1i_1},\dots,\mu_{ni_n}]).\end{align*}

On the other hand,
\begin{align*}\theta_XSP^n_G(\xi_X)& ([M_1,\dots,M_n])=\theta_X([\theta_X(M_1),\dots,\theta_X(M_n)])\\ = & \theta_X([\vee\alpha_{1i_1}\wedge\mu_{1i_1},\dots, \vee\alpha_{1i_1}\wedge\mu_{ni_n}])\\ =& J_\omega(\pi_G)((\vee\alpha_{1i_1}\wedge\mu_{1i_1})\otimes\dots\otimes(\vee\alpha_{1i_1}\wedge\mu_{ni_n}))\\ =& J_\omega(\pi_G)\left(\bigvee(\alpha_{1i_1}\wedge\dots\wedge \alpha_{ni_n})\wedge(\mu_{1i_1}\otimes\dots\otimes\mu_{ni_n})\right) \\ =& \bigvee(\alpha_{1i_1}\wedge\dots\wedge \alpha_{ni_n})\wedge J_\omega(\pi_G)(\mu_{1i_1}\otimes\dots\otimes\mu_{ni_n}),
\end{align*}
 i.e.,  $\xi_XJ_\omega(\theta_X)\theta_{J_\omega(X)}=\theta_XSP^n_G(\xi_X)$. Applying Theorem \ref{ext_kleisli} we obtain that the functor $SP^n_G$ admits an extension onto the Kleisli category of the monad $\mathbb J_\omega$.\end{proof}

\begin{prop} The class of maps $\mathrm{supp}=(\mathrm{supp}_X)\colon J_\omega(X)\to\exp X$ is a morphism of the monad $\mathbb{J}_\omega$ into the hyperspace monad $\mathbb H$.
\end{prop}
\begin{proof} Clearly, for every $x\in X$, where $X$ is an ultrametric space, we have $s_X(x)=\{x\}=\mathrm{supp}(\delta_x)$.

Now let $M\in J_\omega^2(X)$, $M=\vee_{i=1}^n\alpha_i\wedge\mu_i$. We may assume that $\alpha_i>-\infty$, for all $i$. Let also $\mu_i=\vee_{j=1}^{m_i}\beta_{ij}\wedge\delta_{x_{ij}}$, where $\beta_{ij}>-\infty$, for all $i,j$.

Then $\xi_X(M)=\vee_{ij}\alpha_i\wedge\beta_{ij}\wedge\delta_{x_{ij}}$ and  \begin{align*}&u_X\exp(\mathrm{supp}_X)\mathrm{supp}_{J_\omega(X)}(M)=u_X\exp(\mathrm{supp}_X)(\{\mu_1,\dots,\mu_n\})\\ = & u_X\{\{x_{i1},\dots,x_{im_i}\}\mid i=1,\dots,n\}=\{x_{ij}\mid i=1,\dots,n,\ j=1,\dots,m_i\}\\=&\mathrm{supp}(\xi_X(M)).
\end{align*}
\end{proof}

\section{Completion}

Denote by $\cumet$ the category of complete ultrametric spaces and nonexpanding maps. Given a complete ultrametric space $(X,d)$, denote by $J(X)$ the completion of the space $J_\omega X$.

For any morphism $f\colon X\to Y$ of the category $\umet$ there exists a unique morphism $J(F)\colon J(X)\to J(Y)$ that extends $J_\omega(f)$. We therefore obtain a functor  $J\colon\cumet\to\cumet$.

The results of the previous section have their counterpart also for the functor $J$. In particular, we have the following result.

\begin{prop} The functors $I$ and $J$ are isomorphic.
\end{prop}

We keep the notation $\delta_X$ for the natural embedding $x\mapsto\delta_x\colon X\to J(X)$.
Also, for any complete $X$, the set $J_\omega^2(X)$ is dense in $J^2(X)$ and therefore the nonexpanding map $\xi_X\colon J_\omega^2(X)\to J_\omega(X)$ can be uniquely extended to a nonexpanding map $J^2(X)\to J(X)$. We keep the notation $\xi_X$ for the latter map.

\begin{thm} The triple $\mathbb J=(J,\delta,\xi)$ is a monad  in the category $\cumet$.
\end{thm}
\begin{proof} Follows from the proof of Theorem \ref{t:monad}.
\end{proof}
The monad $\mathbb J$ is called the {\em max-min measure monad} in the category $\cumet$.
The support map $$\vee_{i=1}^n\alpha_i\wedge\delta_{x_i}\mapsto \{x_1,\dots,x_n\}\colon J_\omega(X)\to \exp X$$
can be extended to the map $\mathrm{supp}\colon J(X)\to\exp X$, which we also call the support map.

\begin{thm} The class of support maps $J_\omega(X)\to \exp X$ is a morphism of the max-min measure monad to the hyperspace monad  in the category $\cumet$.
\end{thm}

\begin{thm} There exists an extension of the symmetric power functor $SP^n$ onto the  Kleisli category of the monad $\mathbb J$.
\end{thm}

\begin{proof} Similar to the proof of Theorem \ref{t:klei}.

\end{proof}

The category mentioned in the above theorem is nothing but the category of ultrametric spaces and nonexpanding max-min
measure-valued maps.

\begin{thm} The monads $\mathbb I$ and $\mathbb J$ are not isomorphic.
\end{thm}
\begin{proof} This follows from the fact that every morphism of monads generates a morphisms of submonads generated by the subfunctors of finite support.
\end{proof}

\section*{Open problems}

Define the max-min measures for the compact Hausdorff spaces in the spirit of \cite{Z}.
Is the extension of the symmetric power functor $SP^n$ onto the category of ultrametric spaces and max-min-measure-valued maps unique? This is known to be valid for the case of probability measures.

The class of $K$-ultrametric spaces was recently defined and investigated by Savchenko. Can analogs of the results of this paper be proved for the  $K$-ultrametric spaces? See \cite{Sa} where analogous questions are considered.

\section*{Acknowledgements}

This research was supported by the
Slovenian Research Agency grants P1-0292-0101 and J1-4144-0101.


\end{document}